\documentclass[10pt,oneside]{amsart}

\usepackage[
paper=a4paper,
text={138mm,210mm},centering
]{geometry}

\usepackage{amssymb,amsxtra}
\usepackage{float}
\usepackage[all]{xy}
\usepackage{pb-diagram,pb-xy}
\usepackage{graphicx,color,float}
\usepackage[bookmarks]{hyperref}
\usepackage{pinlabel}
\usepackage[
  textwidth=1.2in,
  backgroundcolor=yellow,
  bordercolor=orange,
  textsize=small
]{todonotes}

\iftrue
\makeatletter
\def\@settitle{%
  \vspace*{-20pt}
  \begin{flushleft}%
    \baselineskip14\p@\relax
    \normalfont\bfseries\LARGE
    \@title
  \end{flushleft}%
}
\def\@setauthors{%
  \begingroup
  \def\thanks{\protect\thanks@warning}%
  \trivlist
  \raggedright
  \large \@topsep30\p@\relax
  \advance\@topsep by -\baselineskip
  \item\relax
  \author@andify\authors
  \def\\{\protect\linebreak}%
  \authors
  \ifx\@empty\contribs
  \else
    ,\penalty-3 \space \@setcontribs
    \@closetoccontribs
  \fi
  \normalfont
  \@setaddresses
  \endtrivlist
  \endgroup
}
\def\@setaddresses{\par
  \nobreak \begingroup
  \small
  \def\author##1{\nobreak\addvspace\smallskipamount}%
  \def\\{\unskip, \ignorespaces}%
  \interlinepenalty\@M
  \def\address##1##2{\begingroup
    \par\addvspace\bigskipamount\noindent
    \@ifnotempty{##1}{(\ignorespaces##1\unskip) }%
    {\ignorespaces##2}\par\endgroup}%
  \def\curraddr##1##2{\begingroup
    \@ifnotempty{##2}{\nobreak\noindent\curraddrname
      \@ifnotempty{##1}{, \ignorespaces##1\unskip}\/:\space
      ##2\par}\endgroup}%
  \def\email##1##2{\begingroup
    \@ifnotempty{##2}{\nobreak\noindent E-mail address%
      \@ifnotempty{##1}{, \ignorespaces##1\unskip}\/:\space
      \ttfamily##2\par}\endgroup}%
  \def\urladdr##1##2{\begingroup
    \def~{\char`\~}%
    \@ifnotempty{##2}{\nobreak\noindent\urladdrname
      \@ifnotempty{##1}{, \ignorespaces##1\unskip}\/:\space
      \ttfamily##2\par}\endgroup}%
  \addresses
  \endgroup
  \global\let\addresses=\@empty
}
\def\@setabstracta{%
    \ifvoid\abstractbox
  \else
    \skip@25\p@ \advance\skip@-\lastskip
    \advance\skip@-\baselineskip \vskip\skip@
    \box\abstractbox
    \prevdepth\z@ % because \abstractbox is a vtop
    \vskip-10pt
  \fi
}
\renewenvironment{abstract}{%
  \ifx\maketitle\relax
    \ClassWarning{\@classname}{Abstract should precede
      \protect\maketitle\space in AMS document classes; reported}%
  \fi
  \global\setbox\abstractbox=\vtop \bgroup
    \normalfont\small
    \list{}{\labelwidth\z@
      \leftmargin0pc \rightmargin\leftmargin
      \listparindent\normalparindent \itemindent\z@
      \parsep\z@ \@plus\p@
      
    }%
    \item[\hskip\labelsep\bfseries\abstractname.]%
}{%
  \endlist\egroup
  \ifx\@setabstract\relax \@setabstracta \fi
}
\def\section{\@startsection{section}{1}%
  \z@{-1.2\linespacing\@plus-.5\linespacing}{.8\linespacing}%
  {\normalfont\bfseries\Large}}
\def\subsection{\@startsection{subsection}{2}%
  \z@{-.8\linespacing\@plus-.3\linespacing}{.3\linespacing\@plus.2\linespacing}%
  {\normalfont\bfseries}}
\def\subsubsection{\@startsection{subsection}{3}%
  \z@{.7\linespacing\@plus.2\linespacing}{-1.5ex}%
  {\normalfont\itshape}}
\def\@secnumfont{\bfseries}
\makeatother
\fi %\iftrue/false for amsart.cls hack

\def\to{\mathchoice{\longrightarrow}{\rightarrow}{\rightarrow}{\rightarrow}}
\makeatletter
\newcommand{\shortxra}[2][]{\ext@arrow 0359\rightarrowfill@{#1}{#2}}
\def\longrightarrowfill@{\arrowfill@\relbar\relbar\longrightarrow}
\newcommand{\longxra}[2][]{\ext@arrow 0359\longrightarrowfill@{#1}{#2}}

\makeatother

\def\otimesover#1{\mathbin{\mathop{\otimes}_{#1}}}

\makeatletter
\def\Nopagebreak{\@nobreaktrue\nopagebreak}
\makeatother

\theoremstyle{plain}
\newtheorem{maintheorem}{Theorem}
\newtheorem{theorem}{Theorem}[section]

\newtheorem{lemma}[theorem]{Lemma}

\theoremstyle{definition}
\newtheorem{definition}[theorem]{Definition}

\def\Z{\mathbb{Z}}
\def\Q{\mathbb{Q}}

\def\G{\Gamma}

\def\cP{\mathcal{P}}

\def\a{\alpha}

\def\Ker{\operatorname{Ker}}

\def\Im{\operatorname{Im}}

\def\Bl{B\ell}
\def\lk{\operatorname{lk}}
\def\Arf{\operatorname{Arf}}

\def\rhot{\rho^{(2)}}

\let\oldsharp=\# \def\#{\mathbin{\oldsharp}}

\def\emptystr{}
\newcommand{\mkc}[2][]{\begin{color}{red}#2%
  \def\tempstr{#1}%
  \ifx\tempstr\emptystr \else\textsf{\SMALL\ \raise.7ex\hbox{[\tempstr]}}\fi
\end{color}}

\begin{document}

\title%
[Concordance of links with identical Alexander invariants]
{Concordance of links with identical Alexander invariants}

\author{Jae Choon Cha}
\address{
  Department of Mathematics\\
  POSTECH\\
  Pohang 790--784\\
  Republic of Korea\\
  and\linebreak
  School of Mathematics\\
  Korea Institute for Advanced Study \\
  Seoul 130--722\\
  Republic of Korea
}
\email{jccha@postech.ac.kr}

\author{Stefan Friedl}
\address{
  Mathematisches Institut\\
  Universit\"at zu K\"oln\\
  50931 K\"oln\\
  Germany}
\email{sfriedl@gmail.com}

\author{Mark Powell}
\address{
  Department of Mathematics\\
  Indiana University \\
  Bloomington, IN 47405\\
  USA
}
\email{macp@indiana.edu}

\def\subjclassname{\textup{2010} Mathematics Subject Classification}
\expandafter\let\csname subjclassname@1991\endcsname=\subjclassname
\expandafter\let\csname subjclassname@2000\endcsname=\subjclassname

%\keywords{}

\begin{abstract}
  J. Davis showed that the topological concordance class of a link in
  the 3-sphere is uniquely determined by its Alexander polynomial for
  2-component links with Alexander polynomial one.  A similar result
  for knots with Alexander polynomial one was shown earlier by
  M.~Freedman.  We prove that these two cases are the only exceptional
  cases, by showing that the link concordance class is not determined
  by the Alexander invariants in any other case.
\end{abstract}

\maketitle

\section{Introduction}

J. Davis proved that if a 2-component link $L$ has the Alexander
polynomial of the Hopf link, namely $\Delta_L=1$, then $L$ is
topologically concordant to the Hopf link~\cite{Davis:2006-1}.  In
other words, for 2-components links, the topological concordance class
is determined by the Alexander polynomial $\Delta_L$ when
$\Delta_L=1$.  A natural question arises from this: \emph{for which
  links does the Alexander polynomial determine the topological
  concordance class?}

The answer for knots is already known.  A well-known result of
M.~Freedman (see \cite{Freedman:1984-1},
\cite[11.7B]{Freedman-Quinn:1990-1}) says that it holds for Alexander
polynomial one knots, and T.~Kim~\cite{Kim:2005-1} (extending earlier
work of C.~Livingston~\cite{Livingston:2002-1}) showed that it does
not hold for any Alexander polynomial which is not one.

The following main result of this note says that the results of
M.~Freedman and J.~Davis are the only cases for which the topological
link concordance class is determined by the Alexander polynomial.

\begin{maintheorem}
  \label{maintheorem:simple}
  Suppose $L$ is an $m$-component link, $m\ge 2$, and suppose
  $\Delta_L\ne 1$ if $m=2$.  Then there are infinitely many links
  $L=L_0,L_1,L_2,\ldots$ which have the same  Alexander polynomial
  but are mutually not topologically concordant.
\end{maintheorem}

Recall that the multivariable polynomial $\Delta_L =
\Delta_L(x_1,\ldots,x_m)$ is well-defined up to multiplication by $\pm
x_1^{a_1}\cdots x_m^{a_m}$.  In particular, $\Delta_L\ne 1$ means that
$\Delta_L$ is not of the form $\pm x_1^{a_1}\cdots
x_m^{a_m}$.

We remark that $\Delta_L\ne 1$ is automatically satisfied for $m\ge
3$, as a consequence of the Torres condition, which implies that
$\Delta_L(1,\dots,1) = 0$ for $m \ge 3$.  We also remark that in the
smooth category, it is known that the conclusion of
Theorem~\ref{maintheorem:simple} holds even for knots and 2-component
links with Alexander polynomial one.  The knot case has been
extensively studied in the literature; see, for example,
\cite{Gompf:1986-1,Endo:1995-1} as early works.  The case of links
with unknotted components has been shown recently
in~\cite{Cha-Kim-Ruberman-Strle:2010-01}.

In fact, we can say more about the links $L_i$ in
Theorem~\ref{maintheorem:simple}\hbox{}.  To state the full result, we recall
some terminology in the following two paragraphs.

For an $m$-component link $L$ in $S^3$, denote its exterior by $E_L =
S^3-\nu(L)$ where $\nu(L)$ denotes a tubular neighborhood of $L$.  We always
identify the boundary $\partial E_L$ with $m(S^1\times S^1)$ along the
zero-framing, and view $E_L$ as a bordered
 3-manifold with this marking.

The notions of symmetric grope and Whitney tower
concordance provide a framework for the study of
link concordance.  They measure the failure of links to be concordant in terms of
fundamental geometric constructions, namely gropes and Whitney towers,
in dimension~4.  Roughly speaking, one defines a \emph{height $n$
  (symmetric) Whitney tower concordance} by replacing the embedded
annuli in the definition of concordance with transversely immersed
annuli which form base surfaces supporting a Whitney tower of
height~$n$.  A \emph{height $n$ (symmetric) grope concordance} is
defined similarly by replacing annuli with disjointly embedded height
$n$ gropes.  These were first used in the context of knot slicing by T.~Cochran, K.~Orr, and P.~Teichner~\cite{Cochran-Orr-Teichner:1999-1}.
(Detailed definitions for arbitrary links can be found, for example,
in~\cite[Section~2.4]{Cha:2012-1}.)  Also, in
\cite[Section~2.3]{Cha:2012-1}, the first author introduced an
analogue of these notions for bordered 3-manifolds, which is called an
$n$-solvable cobordism.  Roughly, an \emph{$n$-solvable cobordism} $W$
between bordered 3-manifolds $M$ and $M'$ is a 4-dimensional cobordism
that induces $H_1(M)\cong H_1(W)\cong H_1(M')$ and admits a certain
``lagrangian'' with ``duals'' for the twisted intersection pairing on
$H_2(W;\Z[\pi/\pi^{(n)}])$, where $\pi=\pi_1(W)$ and $\pi^{(n)}$ is the
$n$th derived subgroup (see
Definition~\ref{definition:solvable-cobordism}).

We can now state the full version of our main theorem.

\begin{maintheorem}\label{maintheorem:full}
  Suppose $L_0$ is an $m$-component link, and suppose $\Delta_L\ne 1$
  if $m=2$.  Then there are infinitely many links $L_1,L_2,\ldots$
  satisfying the following:
  \newdimen\templen{\setbox0=\hbox{($1'$)}\global\templen=\wd0}
  \begin{enumerate}
  \item[\hbox to\templen{($1$)\hss}] For each $i$, there is a
    $\Z[\Z^m]$-homology equivalence of $f\colon (E_{L_i},\partial
    E_{L_i}) \to (E_{L_0}, \partial E_{L_0})$ rel $\partial$, namely
    $f|_\partial$ is the identification under the zero-framing and
    \[
    f_*\colon H_*(E_{L_i};\Z[\Z^m]) \to H_*(E_{L_0};\Z[\Z^m])
    \]
    is an isomorphism.
  \item[\hbox to\templen{($1'$)\hss}] The following invariants are
    identical for all the~$L_i$: Alexander polynomial, Alexander
    ideals, Blanchfield form \cite{Blanchfield:1957-1}, Milnor's
    $\overline\mu$-invariants \cite{Milnor:1957-1}, Orr's transfinite
    homotopy invariant $\theta_\omega$ \cite{Orr:1989-1}
    \textnormal{(}whenever defined\textnormal{)}, and Levine's
    homotopy invariant~$\theta$~\cite{Levine:1989-1}
    \textnormal{(}whenever defined\textnormal{)}q.
  \item[\hbox to\templen{($2$)\hss}] For any $i\ne j$, the exteriors
    $E_{L_i}$ and $E_{L_j}$ are not 2-solvably cobordant.
  \item[\hbox to\templen{($2'$)\hss}] For any $i\ne j$, the links
    $L_i$ and $L_j$ are not height 4 grope concordant, not height 4
    Whitney tower concordant, and not concordant.
  \end{enumerate}
\end{maintheorem}

As references for the Alexander invariants appearing in
Theorem~\ref{maintheorem:full}, see, for example,
\cite{Kawauchi:1996-1,Hillman:2002-1}.

Experts will easily see that Theorem~\ref{maintheorem:full} ($1'$)
and ($2'$) are consequences of ($1$) and ($2$), respectively.  In
Section~\ref{section:observations-on-main-theorem}, we discuss this in
more details including some background for the reader's convenience.

We remark that the links $L_i$ in Theorem~\ref{maintheorem:full} can
be chosen in such a way that they are indistinguishable to the eyes of
the asymmetric Whitney tower/grope theory, which is another
framework for the study of link concordance extensively investigated in
recent work of J. Conant, R. Schneiderman, and P. Teichner
(see~\cite{Conant-Teichner-Schneiderman:2011-1} as an extended summary
providing other references).  Namely, the $L_i$ (with the
zero-framing) are mutually order $n$ Whitney tower/grope concordant
for any~$n$ in the sense of
\cite[Definition~3.1]{Conant-Schneiderman-Teichner:2012-2}.  This is
discussed in
Section~\ref{section:satellite-and-asymmetric-Whitney-towers}.

A key ingredient that we use to distinguish concordance classes of
links is the Amenable Signature Theorem which first appeared in
\cite{Cha-Orr:2009-01}.  It generalizes a result presented earlier in
the influential work of
Cochran-Orr-Teichner~\cite{Cochran-Orr-Teichner:1999-1}.  In
\cite{Cha:2012-1}, the first author formulated a symmetric Whitney
tower/grope framework for arbitrary links and bordered 3-manifolds,
and gave (a refined version of) the Amenable Signature Theorem as a
Cheeger-Gromov $\rhot$-invariant obstruction to the existence of
certain Whitney towers and gropes.  In the proof of our main result,
we use a special case of this which is stated as
Theorem~\ref{theorem:amenable-signature-theorem} in this paper.

An interesting aspect of the proof of Theorem~\ref{maintheorem:full} is
that it is separated into two cases which illustrate significantly
different aspects contained in the single problem, as discussed below.

For a 3-manifold $M$, the Cheeger-Gromov $\rhot$-invariant
$\rhot(M,\phi)$ is a real number associated to a group homomorphism
$\phi\colon \pi_1(M)\to \Gamma$, which we call a representation
into~$\Gamma$.  We refer to
\cite{Cheeger-Gromov:1985-1,Cochran-Teichner:2003-1} for details.  An
essential requirement for the use of $\rhot$-invariants in the study
of concordance and related 4-dimensional equivalence relations is that
the representation $\phi$ should have two properties: (1) $\phi$ does
not annihilate certain interesting elements so that it does not lose
too much information, and (2) $\phi$ factors through the fundamental
group of a relevant 4-manifold, for example, the exterior of a
concordance, or more generally, a 4-manifold obtained by symmetric
surgery on a Whitney tower or a grope.

To find such representations, first we consider links which are
``big'' in the
sense that they admit representations into non-abelian nilpotent
quotients.  It is straightforward to show (see
Lemma~\ref{lemma:rank-of-3rd-lower-central-quotient}) that a link $L$
is ``big'' if and only if $L$ has either at least 3 components or if
$L$ is a 2-component link $L$ with $\lk(L)\ne \pm 1$.  For these
links, we apply Dwyer's theorem to show that representations into
certain nilpotent quotients have all the desired properties.  (See
Theorem~\ref{theorem:main-result-for-nontrivial-lcs}.)

For links which are not big, we employ another approach using the
\emph{Blanchfield duality} of the link module $H_1(E_L;\Z[\Z^m])$.  In
fact, for links which admits a nonzero Blanchfield pairing on the
torsion part of the link module, this enables us to prove
Theorem~\ref{maintheorem:full} using certain representations into
\emph{solvable} groups, which are not necessarily nilpotent.  (See
Theorem~\ref{theorem:main-result-for-nonzero-blanchfield}.)  This
applies especially to 2-component links with $\lk(L)\ne 0$ and
$\Delta_L\ne 1$, which we may call ``small''.  (See
Lemma~\ref{lemma:2-comp-link-with-nonzero-blanchfield}).  The case of
small links resembles known approaches to the study of knot
concordance \cite{Casson-Gordon:1986-1,Cochran-Orr-Teichner:1999-1}
and it is related to earlier work of the authors
\cite{Friedl-Powell:2011-1,Cha:2012-1}.

The proofs of the nilpotent and solvable cases of
Theorem~\ref{maintheorem:full} occupy
Sections~\ref{section:links-with-nontrivial-lower-central-quotients}
and~\ref{section:2-comp-links-with-nonzero-lk} respectively.

We remark that on their own, neither the class of ``big'' links nor
the class of ``small'' links covers all the cases in
Theorem~\ref{maintheorem:full}, while they have a significant overlap,
for example two component links $L$ with $|\lk(L)|>1$.  There are
links which do not have useful nilpotent representations, for example
two component links $L$ with $|\lk(L)|=1$, so that the Blanchfield
pairing method is required as discussed above.  On the other hand,
there are links for which the Blanchfield pairing method fails to give
any useful representations.  An enlightening example is the Borromean
rings.  Its Alexander module is generated by the longitudes.  Since
the Blanchfield pairing automatically vanishes on the longitudes, it
is apparent that the Blanchfield pairing cannot be used to prove
property (1) for any representation~$\phi$.

% \begin{figure}[hb]
%   % \labellist
%   % \small\hair 0mm
%   % \pinlabel {$\gamma$} [r] at 230 105
%   % \endlabellist
%   \fbox{Borromean rings and related examples}
%   %\includegraphics[scale=.6]{seed-link}
%   \caption{The Borromean rings and related examples}
%   \label{figure:borromean-rings}
% \end{figure}

\subsubsection*{Conventions}
Manifolds are assumed to be topological and oriented, and submanifolds
are assumed to be locally flat.

\subsubsection*{Acknowledgements}
The authors thank Kent Orr for helpful discussions.  The first author
was supported by NRF grants 2010--0011629, 2010--0029638, and
2012--0009179 funded by the government of Korea.  The second author
would like the Indiana University Mathematics department, and
especially Kent Orr, for its hospitality.

\section{Some observations on Theorem~\ref{maintheorem:full}}
\label{section:observations-on-main-theorem}

In this section we observe that Theorem~\ref{maintheorem:full} ($1'$)
and ($2'$) are consequences of (1) and (2), respectively.  We also
discuss some necessary background.

Recall that a \emph{link} $L$ with $m$ components in $S^3$ is a union
of $m$ disjoint oriented circles embedded in~$S^3$.  If $m=1$ then it
is called a \emph{knot}.  Two links $L$ and $L'$ are \emph{concordant}
if there is an $h$-cobordism (i.e.\ disjoint union of annuli) between
$L\times\{0\}$ and $L'\times\{1\}$ embedded in $S^3\times[0,1]$.

\subsection*{Alexander invariants and Blanchfield pairing}

We will first discuss the Alexander polynomial, Alexander ideals, and
Blanchfield form.  The \emph{Alexander module} of a link $L$ with $m$
components is defined to be $H_1(E_L, \{*\};\Z[\Z^m])$, viewed as a
module over the group ring
$\Z[\Z^m]=\Z[x_1^{\pm1},\ldots,x_m^{\pm1}]$, where the exterior $E_L$
is endowed with the abelianization map $\pi_1(E_L) \to \Z^m$ (sending
the $i$-meridian to the $i$-standard basis vector of $\Z^m$), and $*$
is a fixed basepoint in~$E_L$.  The module $H_1(E_L;\Z[\Z^m])$ is
called the \emph{link module} of~$L$.  The Alexander polynomial and
Alexander ideals are determined by the Alexander module.  It is also
easy to see from the long exact sequence of a pair that the link
module determines the Alexander module and vice versa.  Therefore the
conclusions in Theorem~\ref{maintheorem:full}~($1'$) on the Alexander
polynomials and Alexander ideals are consequences of
Theorem~\ref{maintheorem:full}~($1$).

 Let $Q=\Q(x_1,\ldots,x_m)$ be the quotient field
of~$\Z[\Z^m]$, namely the rational function field on $m$
variables~$x_i$.  For a $\Z[\Z^m]$-module $A$, we denote its torsion
part by
\[
tA = \{x\in A \mid rx = 0 \text{ for some nonzero }r \in \Z[\Z^m]\}.
\]
Due to Blanchfield~\cite{Blanchfield:1957-1}, there is a
sesquilinear pairing
\[
tH_1(E_L;\Z[\Z^m])\times tH_1(E_L;\Z[\Z^m]) \to Q/\Z[\Z^m]
\]
which is called the \emph{Blanchfield pairing} of~$L$.  It is
essentially defined by the duality of $(E_L,\partial E_L)$ over
$\Z[\Z^m]$-coefficients.  We also refer to \cite{Hillman:2002-1},
particularly Section~2.3, for a thorough discussion of the Blanchfield
pairing.

Since it is defined from duality, the Blanchfield pairing is
functorial with respect to maps preserving the fundamental class,
namely degree one maps on link exteriors.  In particular, we have the
following: the conclusion of Theorem~\ref{maintheorem:full}~($1$)
that there is a $\Z[\Z^m]$-homology equivalence $f\colon
(E_{L_i},\partial E_{L_i}) \to (E_{L_0},\partial E_{L_0})$ implies
that the Blanchfield pairings of $L_i$ and $L_0$ are isomorphic.
(Note that a $\Z[\Z^m]$-homology equivalence is automatically an
integral homology equivalence and consequently a degree one map.)

\subsection*{Milnor's invariants}

In \cite{Milnor:1957-1}, J. Milnor defined invariants
$\overline\mu_L(I)$ for a link $L$ with $m$ components, where $I$ is a
finite sequence of integers in $\{1,\ldots,m\}$.  When $I$ has length
$|I|$, $\overline\mu_L(I)$ is called a $\overline\mu$-invariant of
length~$|I|$.  This is the primary invariant for the study of
structure peculiar to link concordance compared to the knot case.

Although $\overline\mu_L(I)$ is originally defined as a certain
residue class of an integer, a known method to formulate that ``two
links have the identical $\overline\mu$-invariants of length $\le q$''
in the strongest sense is as follows.  Recall that the lower central
series of a group $\pi$ is defined by $\pi_1 := \pi$,
$\pi_{q+1}=[\pi,\pi_q]$ where the bracket designates the commutator.
We say that \emph{two links $L$ and $L'$ with $\pi=\pi_1(E_{L})$ and
  $G=\pi_1(E_{L'})$ have the same $\overline\mu$-invariants of length
  $\le q$} if there is an isomorphism $h\colon \pi/\pi_q \to G/G_q$
that preserves (the conjugacy class) of each meridian and each
0-linking longitude.

\begin{lemma}
  Two links $L$ and $L'$ have the same $\overline\mu$-invariants of
  any length if there is an integral homology equivalence $f\colon
  (E_{L},\partial E_{L}) \to (E_{L'},\partial E_{L'})$ rel~$\partial$.
\end{lemma}

\begin{proof}
  Let $\pi=\pi_1(E_{L_i})$ and $G=\pi_1(E_{L_0})$.  By Stallings'
  theorem \cite{Stallings:1965-1}, $f$ induces an isomorphism $h\colon
  \pi/\pi_q \cong G/G_q$.  Since $f$ is fixed on the boundary, $h$
  preserves the conjugacy classes of meridians and longitudes.
\end{proof}

Since the map $f$ in the conclusion of
Theorem~\ref{maintheorem:full}~($1$) is automatically an integral
homology equivalence, it follows that the Milnor invariant conclusion
in Theorem~\ref{maintheorem:full}~($1'$) is a consequence of
Theorem~\ref{maintheorem:full}~($1$).

\subsection*{Homotopy invariants of Orr and Levine}

In \cite{Orr:1989-1}, K. Orr introduced a homotopy theoretic invariant
of links which is still somewhat mysterious.  For a link $L$, suppose
all $\overline\mu$-invariants vanish.  Then for a fixed homomorphism
of the free group $F$ on $m$ generators into $\pi=\pi_1(E_L)$ that
sends generators to meridians, we obtain an induced isomorphism
$F/F_q\cong \pi/\pi_q$ by Stallings' theorem~\cite{Stallings:1965-1}.
These give rise to $\pi\to \overline F := \varprojlim F/F_q$ and $E_L
\to K(\overline F, 1)$.  Let $K_\omega$ be the mapping cone of the map
$K(F,1)\to K(\overline F,1)$ induced by the inclusion $F\to \overline
F$.  Then it is easily seen that the map $E_L \to K(\overline F,1) \to
K_\omega$ extends to a map $o_L\colon S^3 \to K_\omega$.  Its homotopy
class $\theta_\omega(L):=[o_L] \in \pi_3(K_\omega)$ is Orr's
transfinite homotopy invariant.  It is unknown whether this invariant
can be non-vanishing for links which have all
$\overline{\mu}$-invariants zero.

\begin{lemma}
  If there is an integral homology equivalence $f\colon
  (E_{L},\partial E_{L}) \to (E_{L'},\partial E_{L'})$ rel~$\partial$,
  then $\theta_\omega(L)=\theta_\omega(L')$.
\end{lemma}

\begin{proof}
  Let $\pi=\pi_1(E_{L})$ and $G=\pi_1(E_{L'})$.  Fix a map $\mu\colon
  F\to E_L$ sending generators to meridians.  The map $f\circ
  \mu\colon F \to E_{L'}$ also sends generators to meridians.  Define
  the map $o_L$ and $o_{L'}\colon S^3 \to K_\omega$ as above, using
  $\mu$ and $f\circ \mu$. From the definition of $o_L$, it is easily
  seen that the map $g\colon S^3\to S^3$ obtained by filling in $f$
  with the identity map of a solid torus satisfies $g \circ o_L =
  o_{L'}$.  Since $g$ has degree one, it follows that
  $\theta_\omega(L)=[o_L] = [o_{L'}]=\theta_\omega(L')$.
\end{proof}

This shows that the Orr invariant conclusion in
Theorem~\ref{maintheorem:full}~($1'$) is a consequence of
Theorem~\ref{maintheorem:full}~($1$).  The same argument works for
Levine's homotopy invariant $\theta(L)$ defined
in~\cite{Levine:1989-1}.  We omit details.

\subsection*{Solvable cobordism and Whitney tower/grope concordance}

The notion of an $n$-solvable cobordism used in
Theorem~\ref{maintheorem:full} was formulated in
\cite[Section~2.3]{Cha:2012-1}, as a version of an $n$-solution for
manifolds with boundary.  The notion of an $n$-solution was introduced
by Cochran-Orr-Teichner~\cite{Cochran-Orr-Teichner:1999-1}, and was
generalized by Harvey~\cite{Harvey:2006-1} to the case of links.  For
later use in this paper we describe its precise definition below.
Recall that a bordered 3-manifold $M$ over a surface $\Sigma$ is a
3-manifold with boundary identified with $\Sigma$, and for two
bordered 3-manifolds $M$ and $M'$ over the same surface, a relative
cobordism $W$ is a 4-manifold satisfying $\partial W = M\cup_\partial
-M'$.

\begin{definition}[Solvable cobordism]
  \label{definition:solvable-cobordism}
  We say that a relative cobordism $W$ between bordered 3-manifolds
  $M$ and $M'$ is an \emph{$n$-solvable cobordism} if (i) the
  inclusions induce $H_1(M)\cong H_1(W) \cong H_1(M')$ and (ii) there
  are homology classes $\ell_1,\ldots,\ell_m, d_1,\ldots,d_m\in
  H_2(W;\Z[\pi/\pi^{(n)}])$, where $\pi=\pi_1(W)$, such that the
  $\Z[\pi/\pi^{(n)}]$-valued intersection pairing $\lambda_n$ on
  $H_2(W;\Z[\pi/\pi^{(n)}])$ satisfies $\lambda_n(\ell_i,\ell_j)=0$
  and $\lambda_n(\ell_i,d_j)=\delta_{ij}$.
\end{definition}

In \cite[Section~2]{Cha:2012-1} the following was observed by using
techniques in \cite[Section~8]{Cochran-Orr-Teichner:1999-1}:
\begin{align*}
  \text{$L$ and $L'$ are concordant} &\Longrightarrow
  \text{$L$ and $L'$ are height $n+2$ grope concordant}
  \\
  & \Longrightarrow \text{$L$ and $L'$ are height $n+2$ Whitney tower
    concordant}
  \\
  &\Longrightarrow\text{$E_L$ and $E_{L'}$ are $n$-solvably cobordant}
\end{align*}

For the definitions of Whitney tower and grope concordance, refer to,
for example, \cite[Definition~2.12, Definition~2.14]{Cha:2012-1}.

From the above implications, we see that
Theorem~\ref{maintheorem:full}~($2'$) is an immediate consequence of
Theorem~\ref{maintheorem:full}~($2$).

\section{Links with nontrivial lower central series quotients}
\label{section:links-with-nontrivial-lower-central-quotients}

The goal of this section is to prove the following special case of
Theorem~\ref{maintheorem:full}\hbox{}.  Recall that we denote the lower
central series of a group $\pi$ by~$\{\pi_q\}$.

\begin{theorem}
  \label{theorem:main-result-for-nontrivial-lcs}
  Suppose $L$ is an $m$-component link with $\pi=\pi_1(E_L)$ such that
  $\pi_2/\pi_3 \ne 0$.  Then there are infinitely many links $L=L_0,
  L_1,L_2,\ldots$ such that there is a $\Z[\Z^m]$-homology equivalence
  $f\colon (E_{L_i},\partial E_{L_i}) \to (E_{L_0}, \partial E_{L_0})$
  rel $\partial$ for each $i$ but the exteriors $E_{L_i}$ and
  $E_{L_j}$ are not 2-solvably cobordant for any $i\ne j$.
\end{theorem}

From Theorem~\ref{theorem:main-result-for-nontrivial-lcs} and the
discussions in Section~\ref{section:observations-on-main-theorem}, it
follows that Theorem~\ref{maintheorem:full} holds whenever
$\pi_2/\pi_3\ne 0$.

Before we prove Theorem~\ref{theorem:main-result-for-nontrivial-lcs},
we clarify when the lower central series hypothesis is satisfied.

\begin{lemma}
  \label{lemma:rank-of-3rd-lower-central-quotient}
  Suppose $L$ is an $m$-component link with $\pi=\pi_1(E_L)$, $m\ge
  2$.
  \begin{enumerate}
  \item The abelian group $\pi_2/\pi_3$ has rank $\ge (m-1)(m-2)/2$.
  \item If $m=2$, then $\pi_2/\pi_3\cong \Z/\lk(L)\Z$.
  \end{enumerate}
  Consequently, $\pi_2/\pi_3 \ne 0$ if and only if either (i) $m\ge 3$
  or, (ii) $m=2$ and $\lk(L)\ne \pm1$.
\end{lemma}

\begin{proof}
  Milnor \cite[Theorem~4]{Milnor:1957-1} showed that $\pi/\pi_3$ is
  presented by
  \[
  \pi/\pi_3 = \langle x_1,\ldots,x_m  \mid [x_1,\lambda_1], \ldots,
  [x_m,\lambda_m], F_3 \rangle
  \]
  where $m$ is the number of components of $L$, $F_3$ is the 3rd lower
  central subgroup of the free group $F$ on $x_1,\ldots,x_m$, and
  $\lambda_i$ is an element in $F$ which represents the $i$th
  longitude of $L$ in~$\pi/\pi_3$.  It is well known that $F_2/F_3$ is
  the free abelian group generated by the basic commutators $[x_i,
  x_j]$, $i<j$, by Hall's basis theorem.  Also, we have $\lambda_i
  \equiv \prod_{j\ne i} x_j^{\ell_{ij}}$ mod $F_2$, where $\ell_{ij}$
  is the linking number of the $i$th and $j$th components of~$L$.
  Using the standard identities $[a,bc] \equiv [a,b][a,c]$ mod $F_3$
  and $[a,b]^{-1}=[b,a]$, we obtain
  \[
  [x_i,\lambda_i] %\equiv [x_i, \prod_{j\ne i} x_j^{\ell_{ij}}]
  \equiv
  [x_1,x_i]^{-\ell_{1i}} \cdots [x_{i-1},x_i]^{-\ell_{(i-1)i}}
  [x_i,x_{i+1}]^{\ell_{i(i+1)}} \cdots [x_i,x_m]^{\ell_{im}} \text{
    mod } F_3.
  \]
  From this it follows that $\pi_2/\pi_3$ is given by the abelian
  group presentation with $m(m-1)/2$ generators $v_{ij}=[x_i,x_j]$,
  $1\le i < j \le m$, and the following $m$ relators for
  $i=1,\ldots,m$:
  \[
  -\ell_{i1}v_{1i} -\cdots
  -\ell_{(i-1)i}v_{(i-1)i} + \ell_{i(i+1)}v_{i(i+1)} + \cdots
  \ell_{im}v_{im} = 0
  \]
  Note that the $m$ relators add up to zero.  Therefore the rank of
  $\pi_2/\pi_3$ is at least $m(m-1)/2-(m-1) = (m-1)(m-2)/2$.

  For $m=2$, then we have one generator $v_{12}$ and one relator
  $\ell_{12}v_{12}=0$.  Therefore $\pi_2/\pi_3 \cong \Z/\ell_{12} \Z$.
\end{proof}

In the proof of Theorem~\ref{theorem:main-result-for-nontrivial-lcs}
we will make use of the following definition.

\begin{definition}
  For a group $G$ and a sequence $\cP = (R_1, R_2, \ldots)$ of
  commutative rings $R_i$ with unity, we define the
  \emph{mixed-coefficient lower central series} $\{\cP_q G\}$ by
  $\cP_1 G :=G$ and
  \[
  \cP_{q+1}G := \Ker\Big\{\cP_q G \to \frac{\cP_q G}{[G,\cP_q G]}
  \otimesover{\Z} R_q \Big\}.
  \]
\end{definition}

We remark that $\cP_q G$ is a characteristic normal subgroup of~$G$.
We will also make use of the following result from~\cite{Cha:2012-1}:

  \begin{theorem}
    [{A special case of the Amenable Signature Theorem
      \cite[Theorem~3.2]{Cha:2012-1}}]
    \label{theorem:amenable-signature-theorem}
    Suppose $W$ is a $2$-solvable cobordism between bordered
    3-manifolds $M$ and $M'$. Suppose $\Gamma$ is a group which admits
    a filtration $\{e\}\subset \Gamma' \subset \Gamma$ such that
    $\Gamma/\Gamma'$ is torsion-free abelian and such that $\Gamma'$
    is either torsion-free abelian or an abelian $p$-group for some
    prime $p$.  Then for any $\phi\colon \pi_1(M\cup_\partial -M') \to
    \Gamma$ that extends to $\pi_1(W)$, $\rhot(M\cup_\partial
    -M',\phi)=0$.
  \end{theorem}

\begin{proof}
  First note that $\Gamma$ is a solvable group and therefore amenable.
  Furthermore it follows from \cite[Lemma~6.8]{Cha-Orr:2009-01} that
  $\Gamma$ lies in Strebel's class $D(R)$ \cite{Strebel:1974-1} for
  $R=\Q$ or $R=\Z_p$, with $p$ a prime.  The theorem is now an
  immediate consequence of case III of the Amenable Signature
  Theorem~\cite[Theorem~3.2]{Cha:2012-1}, since $\G^{(2)}= \{e\}$ and
  $W$ is a $2$-solvable cobordism.
\end{proof}

We are now ready to give the proof of
Theorem~\ref{theorem:main-result-for-nontrivial-lcs}.  If the reader
is interested in link concordance only, then in the proof below the
phrase ``$2$-solvable cobordism'' can be safely replaced with
``concordance exterior.''

\begin{proof}[Proof of Theorem~\ref{theorem:main-result-for-nontrivial-lcs}]
  By the hypothesis, there exists a simple closed curve $\alpha$ in
  $E_L$ which is a generator of the abelian
  group~$\pi_2/\pi_3$.  We can and will assume that $\alpha$ is
  unknotted in~$S^3$.  We then choose a prime $p$ which divides the
  order of $\alpha$ in~$\pi_2/\pi_3$; if $\alpha$ has infinite
  order, choose any prime~$p$.

  Let $\cP=(\Q,\Z_p)$, so that $\cP_q G$ is defined for $q=1,2,3$.
  Then, for not only the given link group $\pi=\pi_1(E_L)$ but also
  any $\pi$ with $\pi/[\pi,\pi]$ torsion free, $\cP_2\pi$ is the
  ordinary lower central subgroup $\pi_2=[\pi,\pi]$.  Also,
  $\cP_2\pi/\cP_3\pi \cong (\pi_2/\pi_3)\otimes \Z_p$, a $\Z_p$-vector
  space.  Consequently, for the given $\pi=\pi_1(E_L)$, our $\alpha$
  represents a nonzero element in~$\cP_2\pi/\cP_3\pi \subset
  \pi/\cP_3\pi$, namely an element of order~$p$.

  According to Cheeger and Gromov~\cite[p.~23]{Cheeger-Gromov:1985-1}
  there is a constant $R > 0$ determined by the 3-manifold $E_{L}
  \cup_\partial -E_L$ such that $|\rhot(E_{L} \cup_\partial
  -E_L,\Phi)| < R$ for any homomorphism~$\Phi$.  Let us choose knots
  $J_{i}$ inductively for $i=1,2,\ldots$ in such a way that the
  inequality
  \begin{equation}\label{equ:1}
  |\rhot(J_i,\Z_p)| > R + |\rhot(J_j,\Z_p)|
  \end{equation}
  is satisfied whenever $i>j$.
  Here, given a knot $J$ and $p\in \Z$ we write
  \[
  \rhot(J,\Z_p):=\rhot(\mbox{$0$-framed surgery on $J$},\mbox{unique
    epimorphism onto }\Z_p).
  \]
  For example, the connected sum of a sufficiently large number of
  trefoils can be taken as~$J_i$.  To see this, denote the right
  handed trefoil by $T$ and set $C:=\frac{1}{p}\sum_{r=0}^{p-1}
  \sigma_T(e^{2\pi r\sqrt{-1}/p})$, where $\sigma_T(z)$ denotes the
  Levine-Tristram signature of $T$ corresponding to $z\in S^1$.  It is
  straightforward to see that $C> 0$, since at least one of the
  $\sigma_T(z)$ is positive, and all non-zero $\sigma_T(z)$ have the
  same sign.  If $J$ is the connected sum of $k$ copies of $T$ we see
  that
  \begin{equation}
    \rhot(J,\Z_p)=\frac{1}{p}\sum_{r=0}^{p-1} \sigma_{J}(e^{2\pi
      r\sqrt{-1}/p})
    =k\cdot C.
  \end{equation}
  Here, for the first equality we appeal to
  \cite[Corollary~4.3]{Friedl:2003-4} or equivalently
  \cite[Lemma~8.7~(2)]{Cha-Orr:2009-01} and for the second equality we
  use the additivity of the Levine-Tristram signatures.  If we now
  denote by $J_i$ the connected sum of $i\cdot \lceil {R/C}\rceil$
  copies of $T$, then (\ref{equ:1}) is clearly satisfied.

  We then use the satellite construction to produce a new link $L_i :=
  L(\alpha,J_i)$ by tying the knot $J_i$ into $L$ along the
  curve~$\alpha$.  More precisely, by filling in the exterior
  $E_{\alpha}= S^3-\nu (\alpha)$ with the exterior
  $E_{J_{i}}=S^3-\nu(J_{i})$ along an orientation reversing
  homeomorphism of the boundary torus $\partial\nu(\alpha)
  \to \partial\nu(J_{i})$ that identifies a meridian and 0-linking
  longitude of $\alpha$ with a 0-linking longitude and a meridian of
  $J_{i}$ respectively, we obtain a new 3-manifold which is
  homeomorphic to $S^3$, and the image of $L\subset E_{\alpha}$ under
  this homeomorphism is the new link $L_i = L(\alpha,J_{i})$.  We
  denote a 0-framed push-off of $\alpha$ in $E_{L \cup \alpha} \subset
  E_{L_i}$ by~$\alpha_i$.

  It is well known that there is an integral homology equivalence
  $f\colon (E_{L_i},\partial E_{L_i}) \to (E_{L},\partial E_{L})$ (see
  e.g.\ \cite[Lemma~5.3]{Cha-Orr:2011-01}).  In fact $f$ is obtained
  by gluing the identity map on $E_{L\cup \alpha}$ with the standard
  homology equivalence $(E_{J_{i}},\partial E_{J_{i}}) \to
  (E_{J_{0}},\partial E_{J_{0}})=S^1\times (D^2,S^1)$.  Since $\alpha$
  lies in $[\pi_1(E_L),\pi_1(E_L)]$, a Mayer-Vietoris argument applied
  to the above construction shows that $f$ induces isomorphisms on
  $H_*(-;\Z[\Z^m])$.

  By Stallings' theorem \cite{Stallings:1965-1} and our above
  discussion on $\cP_q\pi/\cP_{q+1}\pi$ for $q< 3$, we have an induced
  isomorphism $\pi_1(E_{L_i})/\cP_3 \pi_1(E_{L_i})\cong \pi/\cP_3
  \pi$.  Since $f$ restricts to the identity on $E_{L\cup\alpha}$, the
  element $\alpha_i$ corresponds to $\alpha$ under this isomorphism.

  We will need the following lemma which is a consequence of Dwyer's
  Theorem~\cite{Dwyer:1975-1}, a generalization of Stallings' Theorem.
  We remark that for the special case of a concordance exterior,
  Stallings' Theorem can be used instead.

  \begin{lemma}
    \label{lemma:dwyer-for-solvable-cobordism}
    If $W$ is a $1$-solvable cobordism between two bordered
    3-manifolds $M$ and $M'$ with torsion-free $H_1(M)$, then the
    inclusions induce isomorphisms
    \[
    \pi_1(M)/\cP_{3}\pi_1(M) \cong
    \pi_1(W)/\cP_{3}\pi_1(W) \cong
    \pi_1(M')/\cP_{3}\pi_1(M').
    \]
  \end{lemma}

  \begin{proof}
    Recall Dwyer's theorem~\cite{Dwyer:1975-1}: if $f\colon X\to Y$
    induces an isomorphism $H_1(X)\cong H_1(Y)$ and an epimorphism
    \[
    H_2(X)\to H_2(Y)/\Im \{ H_2(Y;\Z[\pi_1(W)/\pi_1(W)_q]) \rightarrow
    H_2(Y)\},
    \]
    then $f$ induces an isomorphism $\pi_1(X)_q/\pi_1(X)_{q+1} \cong
    \pi_1(Y)_q/\pi_1(Y)_{q+1}$.

    In our case, by the definition of an $n$-solvable cobordism, we
    have $H_1(M)\cong H_1(W)\cong H_1(M')$.  Also, there are
    $1$-lagrangian elements $\ell_1,\ldots,\ell_r$ with $1$-duals
    $d_1,\ldots,d_r$ lying in $H_2(W;\Z[\pi_1(W)/\pi_1(W)^{(1)}])$
    such that the $\ell_i$ and $d_j$ generate~$H_2(W)$.  Since
    $\pi_1(W)^{(1)}$ is equal to $\pi_1(W)_{2}$, the $H_2$ condition
    of Dwyer's theorem is satisfied.  Therefore it follows that
    \[
    \pi_1(M)_q/\pi_1(M)_{q+1} \cong
    \pi_1(W)_q/\pi_1(W)_{q+1} \cong
    \pi_1(M')_q/\pi_1(M')_{q+1}
    \]
    for $q=1, 2$ by Dwyer's theorem.  By our observation that $\cP_2
    \pi=\pi_2$ and $\cP_2\pi/\cP_3\pi =\pi_2/\pi_3 \otimes_{\Z} \Z_p$
    for groups $\pi$ with torsion free $H_1$, we obtain
    \[
    \cP_q\pi_1(M)/\cP_{q+1}\pi_1(M) \cong
    \cP_q\pi_1(W)/\cP_{q+1}\pi_1(W) \cong
    \cP_q\pi_1(M')/\cP_{q+1}\pi_1(M')
    \]
    for $q=1,2$.  From this the desired conclusion follows by the five
    lemma.
  \end{proof}

  Returning to the proof of
  Theorem~\ref{theorem:main-result-for-nontrivial-lcs} let $W$ be a
  2-solvable cobordism between $E_{L_i}$ and~$E_{L_j}$. We will show
  that $i=j$.  First note that we obtain
  \[
  \pi/\cP_3\pi \cong \pi_1(E_{L_i})/\cP_3\pi_1(E_{L_i}) \cong
  \pi_1(W)/\cP_3\pi_1(W) \cong \pi_1(E_{L_j})/\cP_3\pi_1(E_{L_j})
  \]
  by Lemma~\ref{lemma:dwyer-for-solvable-cobordism}.  Let $\phi\colon
  \pi_1(W) \to \Gamma:=\pi_1(W)/\cP_3\pi_1(W)$ be the projection, and
  by abuse of notation, we denote its restriction to $\partial
  W=E_{L_i} \cup_\partial -E_{L_j}$ by $\phi$ as well.  Since $\alpha$
  represents an order $p$ element in $\pi/\cP_3\pi$, both
  $\phi([\alpha_i])$ and $\phi([\alpha_j])$ have order~$p$.

  By applying Theorem~\ref{theorem:amenable-signature-theorem} with $n=2$,
  we obtain that 
  \begin{equation}
    \label{equ:rhozero}
    \rhot(E_{L_i} \cup_\partial-E_{L_j},\phi)=0.
  \end{equation}

  On the other hand, note that the map $\phi$ induces a homomorphism
  $\varphi \colon \pi_1(E_L \cup_{\partial} -E_L) \to \Gamma$ as
  follows.  Recall that $E_{L_i} \cup_\partial -E_{L_j}$ is obtained
  from $E_{L} \cup_\partial -E_L$ by satellite constructions using the
  knots $J_{i}$ and~$J_{j}$.  Viewing $E_{J_{i}}$ as a subspace of
  $E_{L_i} \cup_\partial -E_{L_j}$, the homomorphism $\phi$ restricted
  to $\pi_1(E_{J_{i}})$ sends the meridian of $J_{i}$
  to~$\phi([\alpha_{i}])$.  Since $\phi([\alpha_{i}])$ has order $p$
  in the abelian subgroup $\cP_2\pi/\cP_3\pi$ of $\Gamma$, it follows
  that $\phi$ restricted to $E_{J_{i}}$ factors as $\pi_1(E_{J_i}) \to
  \Z \twoheadrightarrow \Z_p \hookrightarrow \Gamma$ where the first
  map is the abelianization.  Similarly for~$J_{j}$.  It follows that
  $\phi$ on $\pi_1(E_{L_i}\cup_\partial -E_{L_j})$ gives rise to a
  homomorphism $\varphi \colon \pi_1(E_{L}\cup_\partial -E_{L}) \to
  \Gamma$.  To see this, observe that we can arrange an element
  $\gamma \in \pi_1(E_L \cup_{\partial} E_L)$ to avoid $\nu(\a_{i})$
  and $\nu(\a_{j})$.  The image $\varphi(\gamma)$ can then be defined
  by $\phi$.  This is well-defined because crossing $\a_{i}$ in
  homotopy of $\gamma$ changes $\gamma$ by a meridian of
  $\nu(\a_{i})$, which in $E_{L_i}$ is attached to a longitude of
  $J_{i}$, and therefore maps trivially under~$\phi$.

  Now, using (i) the additivity of $\rhot$ under satellite
  construction \cite[Proposition~3.2]{Cochran-Orr-Teichner:2002-1},
  (ii) the $L^2$-induction property of $\rhot$
  \cite[Proposition~5.13]{Cochran-Orr-Teichner:1999-1}, and (iii) the
  fact that $\alpha$ represents an element of order~$p$ in $\Gamma$,
  we obtain that
  \begin{equation}\label{equ:3}
    \rhot(E_{L_i} \cup_\partial     -E_{L_j},\phi) =
    \rhot(E_{L} \cup_\partial-E_L,\varphi) + \rhot(J_i,\Z_p) - \rhot(J_j,\Z_p).
  \end{equation}
  It now follows from (\ref{equ:rhozero}) and the choice of $R$ that
  \[
  \big|\,|\rhot(J_i,\Z_p)| - |\rhot(J_j,\Z_p)|\,\big| <R.
  \]
  In light of (\ref{equ:1}) we now see that $i=j$.
 
  Thus, when $i \ne j$, we have shown that the bordered manifolds
  $E_{L_i}$ and $E_{L_{j}}$ are not 2-solvably cobordant.
\end{proof}

\section{Links with nontrivial Blanchfield pairing}
\label{section:2-comp-links-with-nonzero-lk}

According to Lemma \ref{lemma:rank-of-3rd-lower-central-quotient} the
fundamental group of a $2$-component link with linking number equal to
$\pm 1$ admits no non-abelian nilpotent quotients.  The goal of this
section is to provide an alternative approach, using the Blanchfield
duality, to prove Theorem~\ref{maintheorem:full} for such links.

In this section we denote $\Q[\Z^m]$ by~$\Lambda$, where $m$ is
understood to be the number of components of the link.  The ring
$\Q[\Z^m]$ has the same quotient field as $\Z[\Z^m]$, namely the
rational function field $Q=\Q(x_1,\ldots,x_m)$ considered in
Section~\ref{section:observations-on-main-theorem}.  Recall that for a
$\Lambda$-module $A$ we denote the $\Lambda$-torsion submodule by
$tA$.  In what follows $\Bl_L$ denotes the \emph{rational} Blanchfield
form $\Bl_L\colon tH_1(E_{L};\Lambda)\times tH_1(E_{L};\Lambda) \to
Q/\Lambda$.  We remark that this Blanchfield pairing is obtained by
tensoring the integral Blanchfield pairing discussed in
Section~\ref{section:observations-on-main-theorem} with~$\Q$.

\begin{theorem}
  \label{theorem:main-result-for-nonzero-blanchfield}
  Suppose $L$ is an $m$-component link for which the Blanchfield
  pairing $\Bl_L$ is not constantly zero, i.e., $\Bl_L(x,y)\ne 0$ for
  some $x, y \in tH_1(E_L;\Lambda)$.  Then there are infinitely many
  links $L=L_0,L_1,L_2,\ldots$ such that there is a
  $\Z[\Z^m]$-homology equivalence of $f\colon (E_{L_i},\partial
  E_{L_i}) \to (E_{L_0}, \partial E_{L_0})$ rel $\partial$ for each
  $i$ but the exteriors $E_{L_i}$ and $E_{L_j}$ are not 2-solvably
  cobordant \textnormal{(}and consequently the links $L_i$ and $L_j$
  are not concordant\textnormal{)} for any $i\ne j$.
\end{theorem}

Before proving
Theorem~\ref{theorem:main-result-for-nonzero-blanchfield}, we observe
a special case to which
Theorem~\ref{theorem:main-result-for-nonzero-blanchfield} applies.

\begin{lemma}
  \label{lemma:2-comp-link-with-nonzero-blanchfield}
  Suppose $L$ is a 2-component link with $\lk(L)\ne 0$ and
  $\Delta_L\ne 1$.  Then the Blanchfield pairing $\Bl_L$ is not
  constantly zero.
\end{lemma}

\begin{proof}
  Since $\lk(L)\ne 0$, the Blanchfield pairing $\Bl_L$ on
  $H_1(E_{L};\Lambda)$ is nondegenerate by
  Levine~\cite[Theorem~B]{Levine:1982-1}.  Therefore it suffices to
  show that $H_1(E_{L};\Lambda)$ is nonzero.

  Recall that the Torres condition (see
  e.g.~\cite[Section~5.1]{Hillman:2002-1}) implies that, up to
  multiplication by a monomial, the following equality holds:
  \[
  \Delta_L(x_1,1)=(x_1^{|\lk(L)|-1}+\cdots+x_1+1)\Delta_{K}(x_1).
  \]
  Here $K$ is the first component of $L$. In particular we see that
  $\Delta_L(1,1)=|\lk(L)|$.  Our assumptions that $\lk(L)\ne 0$ and
  $\Delta_L\ne 1$ now immediately imply that $\Delta_L$ is not a
  monomial.  It follows that $H_1(E_{L};\Lambda)$ is nonzero.
\end{proof}

From Theorem~\ref{theorem:main-result-for-nonzero-blanchfield},
Lemma~\ref{lemma:2-comp-link-with-nonzero-blanchfield} and the
discussions in Section~\ref{section:observations-on-main-theorem}, it
follows that Theorem~\ref{maintheorem:full} holds for 2-component
links with nonzero linking number and with $\Delta_L\ne 1$.  This,
combined with Theorem~\ref{theorem:main-result-for-nontrivial-lcs} and
Lemma~\ref{lemma:rank-of-3rd-lower-central-quotient}, completes the
proof of Theorem~\ref{maintheorem:full}, modulo the proof of
Theorem~\ref{theorem:main-result-for-nonzero-blanchfield} which is
given below.

\begin{proof}[Proof of
  Theorem~\ref{theorem:main-result-for-nonzero-blanchfield}]
  We choose simple closed curves $\alpha_1,\dots,\alpha_N$ in $E_{L}$
  which are unknotted in $S^3$ and have linking number zero with $L$
  such that their classes $[\alpha_k]$ generate $tH_1(E_{L};\Lambda)$,
  which is a finitely generated module since $\Lambda$ is Noetherian.
  For each $i=1,2,\ldots,$ we use the satellite construction to
  produce a new link $L_i = L(\{\alpha_k\},\{J_{ik}\})$ by tying a
  collection of knots $\{J_{ik}\}_{k=1}^{N}$ into $L$ along the
  curves~$\alpha_k$, for $k=1,\dots,N$.  (See the proof of
  Theorem~\ref{theorem:main-result-for-nontrivial-lcs} for a more
  detailed description of the satellite construction.)  We define
  $J_{0k}$ to be the trivial knot for each $k$, so that
  $L=L_0=L(\{\alpha_k\}, \{J_{0k}\})$ is also described in the same
  way.  For $i\ge 1$ we choose the knots $J_{ik}$ as follows.  As made
  use of in the proof of Theorem
  \ref{theorem:main-result-for-nontrivial-lcs}, due to
  Cheeger-Gromov~\cite[p.~23]{Cheeger-Gromov:1985-1}, there is a
  constant $C > 0$ determined by the 3-manifold $E_{L} \cup_\partial
  -E_L$ such that $|\rhot(E_{L} \cup_\partial -E_L,\varphi)| < C$ for
  any representation~$\varphi$.  For a knot $K$, we define
  $\rhot(K)=\int \sigma_K(\omega) \, d\omega$, the integral of the
  Levine-Tristram signature function $\sigma_K(\omega)$ over the unit
  circle normalized to length one.  For example, an elementary
  calculation shows that if $K$ is the trefoil, then
  $\rhot(K)=\frac{4}{3}$ (see
  e.g.~\cite{Cochran-Orr-Teichner:1999-1}).  We choose knots $J_{i}$
  inductively for $i=1,2,\ldots$ in such a way that the inequality
  \begin{equation}
    \label{equation:defining-condition-of-J_i}
    |\rhot(J_i)| > C + N|\rhot(J_j)|
  \end{equation}
  is satisfied whenever $i>j$; recall that $N$ is the number of
  satellite curves~$\a_k$.  Once again, the connected sum of a
  sufficiently large number of trefoils can be taken as~$J_i$.  Let
  $J_{ik}=J_i$ for every~$k$.

  As we observed in the proof of
  Theorem~\ref{theorem:main-result-for-nontrivial-lcs}, there is an
  integral homology equivalence $f\colon (E_{L_i},\partial E_{L_i})
  \to (E_L,\partial E_L)$.  Since each $\alpha_k$ lies in
  $[\pi_1(E_L),\pi_1(E_L)]$, a Mayer-Vietoris argument applied to our
  satellite construction shows that $f$ induces isomorphisms on
  $H_*(-;\Z[\Z^m])$.

  Let $\alpha_{ik}\subset E_{L\cup \alpha}\subset E_{L_i}$ be a
  push-off of $\alpha_k$ along the zero-framing.  By the above, the
  $[\alpha_{ik}]$ generate $H_1(E_{L_i};\Lambda) \cong
  H_1(E_{L};\Lambda)$.  As we discussed in
  Section~\ref{section:observations-on-main-theorem}, $\Bl_{L_i} \cong
  \Bl_L$ under the map induced by $f$ since $f$ is a
  $\Q[\Z^m]$-homology isomorphism.

  Suppose $W$ is a 2-solvable cobordism between $E_{L_i}$ and
  $E_{L_j}$ for some $i\geq j$. We will show that this implies that
  $i=j$.  We need the following fact, which is obtained immediately by
  combining \cite[Theorem~4.12]{Cha:2012-1} with
  \cite[Corollary~4.14]{Cha:2012-1}, the proofs of which relied on
  arguments due
  to~\cite[Theorem~4.4~and~Lemma~4.5]{Cochran-Orr-Teichner:1999-1}
  and~\cite[Lemma~5.10]{Cochran-Harvey-Leidy:2009-01}.

  \begin{theorem}[\cite{Cha:2012-1}]
    \label{theorem:blanchfield-self-annilator}
    Suppose that $W$ is a 1-solvable cobordism between link exteriors
    $E_{L}$ and~$E_{L'}$.  Then the submodule
    $P=\Ker\{tH_1(E_L;\Lambda)\to tH_1(W;\Lambda)\subset
    H_1(W;\Lambda)\}$ satisfies $\Bl_L(P,P)=0$.
  \end{theorem}

  In our case, from Theorem~\ref{theorem:blanchfield-self-annilator}
  and the hypothesis that $\Bl_L\cong \Bl_{L_i}$ is not constantly
  zero, it follows that $P=\Ker\{tH_1(E_{L_i};\Lambda)\to
  tH_1(W;\Lambda)\}$ is not equal to $tH_1(E_{L_i};\Lambda)$.  That
  is, $[\alpha_{ik}]\not\in P$ for some~$k$.

  For a group $G$, we now denote by $\{\cP^n G\}$ the
  rational derived
  series, namely $\cP^0 G :=G$ and
  \[
  \cP^{n+1}G := \Ker\Big\{\cP^n G \to \frac{\cP^n G}{[\cP_n G,\cP_n G]}
  \otimesover{\Z} \Q \Big\}.
  \]
  Then for $\pi= \pi_1(W)$, $\cP^1\pi$ is the ordinary commutator
  subgroup $\pi^{(1)}=[\pi,\pi]$ since the abelian group
  $\pi/[\pi,\pi]$ is torsion free.  Also, $\cP^1\pi/\cP^2\pi$ is the
  quotient of $\pi^{(1)}/\pi^{(2)} = H_1(W;\Z[\Z^m])$ by its
  $\Z$-torsion subgroup.

  Let $\phi\colon \pi\to \Gamma := \pi/\cP^2\pi$ be the projection,
  and by abuse of notation, we denote its restriction to $\partial
  W=E_{L_i} \cup_\partial -E_{L_j}$ by $\phi$ as well.  Since $\cP^1
  \pi/\cP^2\pi = H_1(W;\Z[\Z^2])/\text{($\Z$-torsion)}$ injects into
  $H_1(W;\Lambda)$ and the image of $[\alpha_{ik}]$ in
  $H_1(W;\Lambda)$ is nontrivial for some $k$, it follows that
  $\phi([\alpha_{ik}])$ is nontrivial for some~$k$.  Furthermore,
  $\phi([\alpha_{ik}])$ has infinite order, since
  $\phi([\alpha_{ik}])$ lies in $\cP^1\pi/\cP^2\pi$ which is a
  torsion-free abelian group.  Similarly $\phi([\alpha_{jk}])$ has
  infinite order for some (possibly different)~$k$.

  Theorem~\ref{theorem:amenable-signature-theorem} applies to the
  group $\Gamma$, with $\Gamma' = \cP^1\pi/\cP^2\pi$.  We thus deduce
  that $\rhot(E_{L_i} \cup_\partial -E_{L_j}, \phi)=0$.

  The homomorphism $\phi$ on $\pi_1(E_{L_i}\cup_\partial -E_{L_j})$
  gives rise to a homomorphism $\varphi \colon
  \pi_1(E_{L}\cup_\partial -E_{L}) \to \Gamma$, which is defined in a
  very similar way to the map which was also called $\varphi$ in the
  proof of Theorem \ref{theorem:main-result-for-nontrivial-lcs}: the
  homomorphism $\phi$ restricted to $\pi_1(E_{J_{ik}})$ sends the
  meridian of $J_{ik}$ to~$\phi([\alpha_{ik}])$.  Since
  $\phi([\alpha_{ik}])$ has infinite order in the abelian subgroup
  $\cP^1\pi/\cP^2\pi$ of $\Gamma$, it follows that $\phi$ restricted
  to $E_{J_{ik}}$ factors as $\pi_1(E_{J_i}) \to \Z \hookrightarrow
  \Gamma$ where the first map is the abelianization.  Similarly
  for~$J_{jk}$.  It follows that $\phi$ on $\pi_1(E_{L_i}\cup_\partial
  -E_{L_j})$ gives rise to a homomorphism of $\pi_1(E_{L}\cup_\partial
  -E_{L})$, say~$\varphi$.

  By the above and by the same argument as in the proof of
  Theorem~\ref{theorem:main-result-for-nontrivial-lcs} we obtain that
  \begin{equation}
    \label{equation:blanchfield-case-rho-formula}
    \begin{aligned}
      0 & = \rhot(E_{L_i} \cup_\partial -E_{L_j},\phi)
      \\
      & =\rhot(E_{L} \cup_\partial -E_L,\varphi) +
      \sum_{k=1}^N\epsilon_{ik}\rhot(J_{ik},\psi) -
      \sum_{k=1}^N\epsilon_{jk}\rhot(J_{jk},\psi)
    \end{aligned}
  \end{equation}
  where $\psi$ denotes the abelianization epimorphism of a knot group
  onto $\Z$ and where $\epsilon_{ik}$ is $0$ if $\phi(\alpha_{ik})$ is
  trivial, $1$ otherwise, and $\epsilon_{ij}$ is defined similarly.
  Furthermore, by \cite[Proposition~5.1]{Cochran-Orr-Teichner:2002-1}
  we know that for any knot $K$ the invariant $\rhot(K,\psi)$ is equal
  to the integral of the Levine-Tristram signature function.  Since
  $\epsilon_{ik}=1$ for some $k$ (and similarly for $\epsilon_{jk}$)
  and our $J_i$ are chosen so that the
  inequality~\eqref{equation:defining-condition-of-J_i} is satisfied
  whenever $i>j$, the
  equality~\eqref{equation:blanchfield-case-rho-formula} can only be
  satisfied if $i=j$.  From this it follows that there is no
  2-solvable cobordism between $E_{L_i}$ and~$E_{L_j}$ whenever $i\ne
  j$.
\end{proof}

\section{Satellite construction and asymmetric Whitney towers}
\label{section:satellite-and-asymmetric-Whitney-towers}

In this section we observe that our links $L_i$ in
Theorem~\ref{maintheorem:full} can be assumed to be mutually order $n$
Whitney tower/grope concordant for any $n$.  For the definition of
order $n$ Whitney tower concordance of framed links, see
\cite[Definition~3.2]{Conant-Schneiderman-Teichner:2012-2}; in this
section we assume that links are always endowed with the zero-framing.

Since we constructed $L_i$ by satellite construction on a given link $L$
using some knots which are only required to have sufficiently large
integral (or average of finitely many evaluations) of the
Levine-Tristram signature, the claim follows immediately from the
following lemma:

\begin{lemma}\label{lemma:asymmetric-Whitney-towers}
  Suppose $K$ is a knot in $S^3$ with vanishing Arf invariant, $L$ is
  a link in $S^3$, and $\alpha$ is a simple closed curve in $S^3-L$
  which is unknotted in~$S^3$.  Then the link $L'=L(\alpha,K)$
  obtained by the satellite construction is order $n$ Whitney
  tower/grope concordant to $L$ for any~$n$.
\end{lemma}

For example, in the construction of our examples above, take an even
number of trefoils for $J_{i}$.  Note that asymmetric Whitney
tower/grope concordance contains no information even when we use
representations to nilpotent groups to obstruct symmetric Whitney
tower concordance.

We remark that in \cite{Schneiderman:2006-1}, Schneiderman showed that
the Whitney tower and grope concordance are equivalent in the
asymmetric case.  Therefore it suffices to show our results for grope
concordance.  A brief outline of the proof is: $K$ bounds an order $n$
grope in $D^4$ since $\Arf(K)=0$, and then a ``boundary connected
sum'' of parallel copies of this grope and the product concordance
from $L$ to $L$ becomes an order $n$ grope concordance between $L$
and~$L'$.  The details are spelled out below.

\begin{proof}[Proof of Lemma \ref{lemma:asymmetric-Whitney-towers}]
  We begin with a well-known description of the satellite
  construction.  Choose an embedded 2-disk $D$ in $S^3$ which is
  bounded by $\alpha$ and meets $L$ transversely.  Choose an open
  regular neighborhood $U$ of $D$ in $S^3$ for which $(U,U\cap L)$ is
  a trivial $r$-string link where $r=|D\cap L|$.  For the knot $K$,
  take the union $Y$ of $r$ parallel copies of $K$ and take an open
  regular neighborhood $V$ of a 2-disk fiber of the normal bundle of
  $K$ such that $(V,V\cap Y)$ is a trivial $r$-string link.  There is
  an orientation reversing homeomorphism $h\colon (U,U\cap L) \to
  (V,V\cap Y)$ such that
  \[
  (S^3,L') = \Big(\big((S^3,L)- (U,U\cap L)\big) \cup
  \big((S^3,Y)-(V,V\cap Y)\big)\Big)\Big/x\sim h(x)\text{ for }
  x\in \partial U.
  \]
  Here components of $Y-V$ are oriented according to the sign of the
  intersection points in~$D\cap L$.

  From our assumption that $\Arf(K)=0$ and the result
  in~\cite{Schneiderman:2005-1} that the Arf invariant is the only
  obstruction to a knot bounding a framed embedded grope of arbitrary
  order, it follows that there is a framed embedded grope of order $n$
  in $D^4$ bounded by~$K$ for any~$n$.  Taking $r$ parallel copies of
  the grope (and orienting the base surfaces according to the sign of
  the intersection points in~$D\cap L$), we obtain a framed embedded
  grope $G$ bounded by~$Y$.

  Identify a collar neighborhood of $S^3$ in $D^4$ with
  $S^3\times[0,\epsilon]$.  We may assume that $V\times [0,\epsilon]$
  intersects (the base surface of) $G$ in $(V\cap Y)\times
  [0,\epsilon]$.  Now, define $G'\subset S^3\times [0,1]$ by forming
  the union
  \begin{align*}
    (S^3\times[0,1],G') &= \Big(\big((S^3,L)\times[0,1]-(U,U\cap
    L)\times[0,\epsilon)\big)
    \\
    &\hphantom{=}\qquad \cup \big((D^4,G)-(V,V\cap
    Y)\times[0,\epsilon)\big)\Big)\Big/\sim
  \end{align*}
  where $(x,t)\sim h(x,t)$ for $(x,t)\in(\partial U\times[0,\epsilon])
  \cup (U\times\{\epsilon\})$.  Then $G'$ is a grope concordance of
  order $n$ cobounded by $L'\times 0$ and $L\times 1$.
\end{proof}

\bibliographystyle{amsalpha} \renewcommand{\MR}[1]{}

\bibliography{research}

\end{document}